\documentclass{amsart}
\usepackage{amsmath, amssymb, mathrsfs, verbatim, multirow}
\usepackage[all]{xy}

\newtheorem{Teo}{Theorem}[section]
\newtheorem{Prop}[Teo]{Proposition}
\newtheorem{Lema}[Teo]{Lemma}
\newtheorem{Cor}[Teo]{Corollary}

\theoremstyle{definition}
\newtheorem{Def}[Teo]{Definition}

\newtheorem{Obs}[Teo]{Remark}

\newtheorem{Exa}[Teo]{Example}
\newtheorem{Pro}[Teo]{Problem}

\newcommand{\R}{\mathbb{R}}

\newcommand{\C}{\mathbb{C}}

\newcommand{\F}{\mathbb{F}}

\newcommand{\lra}{\longrightarrow}

\newcommand{\MI}{\mathfrak{m}}
\newcommand{\QI}{\mathfrak{q}}

\newcommand{\QF}{\mbox{\rm Quot}}

\begin{document}
\title{Topologies on spaces of valuations: a closeness criterion}
\author{Josnei Novacoski}
\address{Josnei Novacoski\newline\indent CAPES Foundation  \newline \indent Ministry of Education of Brazil \newline \indent Bras\'ilia/DF 70040-020 \newline \indent Brazil}
\email{jan328@mail.usask.ca}

\thanks{During part of the realization of this project the author was supported by a graduate research fellowship at University of Saskatchewan. During the other part the support came from a grant from the program ``Ci\^encia sem Fronteiras" from the Brazilian government.}

\keywords{Spaces of valuations}
\subjclass[2010]{Primary 13A18; Secondary 06F05, 06F20}
\begin{abstract}
This paper is part of a program to understand topologies on spaces of valuations. We fix an ordered abelian group $\Gamma$ and an integral domain $R$. We study the relation between a topology on $\Gamma_\infty$ and the induced topology on the set $\mathcal W_\Gamma$ of all valuations of $R$ taking values in $\Gamma_\infty$. For instance, we give a criterion for $ \mathcal W_\Gamma$ to be closed in $\left(\Gamma_\infty\right)^R$. We also discuss the effect of this criterion for natural topologies on $\Gamma_\infty$.
\end{abstract}

\maketitle
\section{Introduction}
Topologies on spaces of valuations appear in many different settings. The most common of them is the Zariski topology. It was introduced by Zariski in the first half of the twentieth century and it has been intensively studied since then. Initially, it was defined as a topology on algebraic varieties, but in a modern language it is defined as a topology on the spectrum of a ring $A$, i.e., the set of all prime ideals of $A$. The space of Krull valuations on a ring (or valuations on a field) admits a natural structure as inverse limit of spectra of rings with their respective Zariski topologies. The corresponding topology is called again the Zariski topology on the space of valuations.

An important application of this topology comes on the problem of resolution of singularities of an algebraic variety. This is known as the Zariski's approach for resolution of singularities. It consists in first proving that every valuation on the function field of the given variety admits local uniformization. Such uniformization holds for an open set on the Zariski topology. This means that, by compactness, in order to resolve all singularities, it is enough to glue only finitely many local solutions. This approach has been very successful. For instance, Zariski proved resolution of singularities for algebraic varieties of dimension up to three in characteristic zero (see \cite{Zar_4}). Another example is the recent work of Cossart and Piltant on resolution of singularities for threefolds. They proved resolution of singularities for any threefold over a field of positive characteristic (see \cite{Cos_2} and \cite{Cos_3}).

A variation of the Zariski topology is the so called the patch topology. This is defined as the patch (or constructive) topology induced by the Zariski topology. It is the coarsest compact and Hausdorff topology finer than the Zariski topology. The Zariski topology and patch topologies are constructed in the set of Krull valuations of a ring. Their definitions cannot be extended directly to the set of all valuations on this ring. To overcome this problem, Huber and Knebusch introduced the valuation spectrum topology (see \cite{Kne}). In the appendix one can find a more detailed discussion as well as the formal definitions.

Berkovich spaces (\cite{Ber}) also appear as an alternative theory. An example of that is the valuative tree, developed by Favre and Jonsson in \cite{Fav_1}. The valuative tree consists of all (normalized and centered) valuations in $\C[[x,y]]$ taking values in $\R$. Favre and Jonsson give a tree structure to this space. These tree structure generates different topologies on this set. Again, we present a discussion on this topic in the appendix.

These topologies defined above are compact. The proofs of their compactness usually follow the following classical argument: we see $\mathcal W$ as a subset of $X^I$ where $X$ is a Hausdorff and compact space and $I$ is an indexing set. By Tychonoff theorem, $X^I$ is also Hausdorff and compact. Then it is enough to prove that $\mathcal W$ is closed in $X^I$.

Motivated by this fact, we consider the following problem. Fix a ring $R$ and an ordered abelian group $\Gamma$. Then the set $\mathcal W_\Gamma$ of all valuations of $R$ taking values in $\Gamma_\infty$ is naturally a subset of $\left(\Gamma_\infty\right)^R$. Given a topology $\mathfrak A$ on $\Gamma_\infty$, when is $\mathcal W_\Gamma$ closed (with respect to the product topology) in $\left(\Gamma_\infty\right)^R$? We prove the following:

\begin{Teo}\label{topvaltheo}
Let $\Gamma'$ be any submonoid of $\Gamma_\infty$ and take a topology $\mathfrak A$ on $\Gamma'$ such that
\begin{description}
\item[(P1)] the addition $+:\Gamma'\times\Gamma'\lra\Gamma'$ is continuous, and
\item[(P2)] for every $\gamma,\gamma'\in\Gamma'$ such that $\gamma<\gamma'$ there exist open sets $U,U'\in\mathfrak A$ such that $\gamma\in U,\gamma'\in U'$ and $U< U'$ (i.e., $u<u'$ for every $u\in U$ and $u'\in U'$).
\end{description}
Then the set $\mathcal{W}_{\Gamma'}$ of valuations of $R$ taking values in $\Gamma'$ is closed in $\left(\Gamma'\right)^{R}$.
\end{Teo}

\begin{Obs}
We introduce the submonoid $\Gamma'$ of $\Gamma_\infty$, and prove Theorem \ref{topvaltheo} for this case, because in many situations we will study valuations which take values in a specific submonoid. For instance, the values of centered valuations (see Section \ref{prelim}) lie in the submonoid $\left(\Gamma^{\geq 0}\right)_\infty$ of $\Gamma_\infty$.
\end{Obs}

\begin{Obs}
The theorem above obviously holds for the case where the trivial valuation is the only valuation in $R$ (i.e., if $R$ is an algebraic extension of $\F_p$ for some prime number $p$). However, we will always assume that this is not the case.
\end{Obs}
We can ask for the converse of Theorem \ref{topvaltheo}, namely: can we find conditions on $\mathfrak A$ which imply that $\mathcal{W}_{\Gamma'}$ is not closed in $(\Gamma')^R$? The next proposition answers this question for the case $\Gamma'=\Gamma_\infty$.

\begin{Prop}\label{Proptop}
Take any topology $\mathfrak A$ on $\Gamma_\infty$. If the set of all valuations on $R$ taking values in $\Gamma_\infty$ is closed in $(\Gamma_\infty)^R$ with respect to the product topology, then $\mathfrak A$ is $T_1$.
\end{Prop}

It is easy to show that if \textbf{(P2)} holds, then $\mathfrak A$ is finer than the order topology and that both \textbf{(P1)} and \textbf{(P2)} hold for the order topology. On the other hand, if $\mathfrak A$ is not $T_1$, then the set of all valuations on $R$ taking values in $\Gamma_\infty$ is not closed in $\left(\Gamma_\infty\right)^R$. A natural question is whether the property of being $T_1$ characterizes the order topology among all the topologies $\mathfrak A$ coarser than the order topology. In Section \ref{proofsresults} we present an example that answers this question to the negative. Namely, we construct a topology on $\Gamma_\infty$ strictly coarser than the order topology which is $T_1$ (in fact, this topology is normal).

This paper is divided as follows. In Section \ref{prelim} we present the basic definitions that will be needed in the sequel. In Section \ref{proofsresults} we present the proofs of Theorem \ref{topvaltheo} and Proposition \ref{Proptop}. In the last section we present some natural topologies in $\Gamma_\infty$ and discuss whether, for each of these topologies, the set $\mathcal W_\Gamma$ is closed in $\left(\Gamma_\infty\right)^R$. We also include an appendix where we present the definitions and comparisons between the known topologies on spaces of valuations.

\par\medskip
\textbf{Acknowledgements.} The author is thankful to Franz-Viktor Kuhlmann for his guidance during the realization of this project. His careful reading and the posterior comments and corrections made this work more consistent and precise.

\section{Preliminaries}\label{prelim}
Take a commutative ring $R$ with unity and an ordered abelian group $\Gamma$. We denote by $\Gamma_\infty$ to the abelian group $\Gamma\cup\{\infty\}$ where $\infty$ is a symbol not belonging to $\Gamma$. We extend addition and order to $\infty$ is as usual.

\begin{Def}
A \index{Valuation}\textbf{valuation} on $R$ taking values in $\Gamma_\infty$ is a mapping $\nu:R\lra \Gamma_\infty$ with the following properties:
\begin{description}
\item[(V1)] $\nu(ab)=\nu(a)+\nu(b)$ for all $a,b\in R$.
\item[(V2)] $\nu(a+b)\geq \min\{\nu(a),\nu(b)\}$ for all $a,b\in R$.
\item[(V3)] $\nu(1)=0$ and $\nu(0)=\infty$.
\end{description}
\end{Def}

The subgroup of $\Gamma$ generated by
\[
\{\nu(a)\mid a\in R\mbox{ and } \nu(a) \neq \infty\}
\]
is called the \index{Value group!of a valuation}\textbf{value group of $\nu$} and is denoted by $\nu R$. The valuation $\nu$ is called \index{Valuation!trivial}\textbf{trivial} if $\nu R = \{0\}$. The set $\mathfrak q_\nu: = \nu^{-1}(\infty)$ is a prime ideal of $R$, called the \index{Valuation!support of }\textbf{support of $\nu$}.

\begin{Def}
A valuation $\nu:R\lra \Gamma_\infty$ is a \index{Valuation!Krull}\textbf{Krull valuation} if $\mathfrak q_\nu=\{0\}$.
\end{Def}

We denote by $\mathcal{W}$ the class of all valuations on $R$. For a fixed ordered abelian group $\Gamma$, we denote by $\mathcal W_\Gamma$ to the set of all valuations taking values in $\Gamma_\infty$. Also, we denote by $\mathcal{V}$ and $\mathcal V_\Gamma$ the subset of $\mathcal{W}$ and $\mathcal V_\Gamma$, respectively, consisting of all Krull valuations. Given a valuation $\nu$ on $R$, we can define a Krull valuation
\[
\overline{\nu}:R/\mathfrak q_\nu\lra \Gamma_\infty
\]
by setting $\overline{\nu}(\overline{a})=\nu(a)$, where $\overline a$ denotes the reduction of $a$ modulo $\QI_\nu$.

\begin{Def}\label{equival}
Two valuations $\nu$ and $\mu$ on $R$ are said to be equivalent (and denote by $\nu\sim\mu$) if one (and hence all) of the following equivalent conditions are satisfied:
\begin{description}
\item[i)] For all $a,b \in R,\ \nu(a) > \nu(b)$ if and only if $\mu(a) >\mu(b)$.
\item[ii)] There is an order-preserving isomorphism $f:\nu R\lra \mu R$ such that $\mu=f\circ \nu$.
\item[iii)] The ideals $\mathfrak q_\nu$ and $\mathfrak q_\mu$ are equal and for any $\overline a/\overline b\in \QF(R/\QI_\nu)=\QF(R/\QI_\mu)$ we have that $\overline\nu(\overline a/\overline b)\geq 0$ if and only if $\overline\mu(\overline a/\overline b)\geq 0$.
\end{description}
\end{Def}

\begin{Obs}
If $\nu$ and $\mu$ are two real valued valuations (i.e., $\nu,\mu\in\mathcal W_\R$), then $\nu\sim \mu$ if and only if $\nu=C \cdot\mu$ for some $C\in \R$ and $C>0$.
\end{Obs}

Take a local ring $(R,\MI)$ and $\nu$ a valuation on $R$. We will say that $\nu$ is \index{Valuation! centered}\textbf{centered} at $R$ if $\nu(a)\geq 0$ for all $a\in R$ and $\nu(a)>0$ for all $a\in \MI$. If in addition $R$ is noetherian, then $\MI$ is finitely generated, so we can define
\[
\nu(\MI):=\min \{\nu(a)\mid a\in \MI\}.
\]
In this case, we denote by $\mathcal W^{\geq 0}$ and $\mathcal V^{\geq 0}$ the subsets of $\mathcal W$ and $\mathcal V$, respectively, consisting all centered valuations on $R$.

\section{Proofs of the main results}\label{proofsresults}
In this section we present the proofs of Theorem \ref{topvaltheo} and Proposition \ref{Proptop}.

\begin{proof}[Proof of Theorem \ref{topvaltheo}]
We will prove that $\left(\Gamma'\right)^R\setminus\mathcal{W}_{\Gamma'}$ is an open set in the product topology. Take a function $f:R\lra\Gamma'$ which is not a valuation. Then one of the three axioms \textbf{(V1)}, \textbf{(V2)} or \textbf{(V3)} does not hold for $f$. We will treat each case separately.

If \textbf{(V1)} does not hold, then $f(ab)\neq f(a)+f(b)$ for some $a,b\in R$. Property \textbf{(P2)} implies that $\mathfrak A$ is Hausdorff, so there exist $U,W\in\mathfrak A$ such that $f(a)+f(b)\in U$, $f(ab)\in W$ and $U\cap W=\emptyset$. By \textbf{(P1)} there exist $V,V'\in\mathfrak A$ with $f(a)\in V$ and $f(b)\in V'$ such that $V+V'\subseteq U$. For an element $a\in R$ we define the map $\pi_a:\left(\Gamma'\right)^{R} \lra \Gamma'$ by $\pi_a(f):=f(a)$.
Take the open set given by
\[
O=\pi_a^{-1}(V)\cap\pi_b^{-1}(V')\cap\pi_{ab}^{-1}(W).
\]
It is easy to see that $f\in O$. Take any element $g\in O$ and let us prove that $g$ is not a valuation. Since $g(a)\in V$ and $g(b)\in V'$ we must have $g(a)+g(b)\in V+V'\subseteq U$. On the other hand, $g(ab)\in W$. This means that $g(ab)\neq g(a)+g(b)$ because $U\cap W=\emptyset$. Hence, $g$ is not a valuation.

If \textbf{(V2)} does not hold, then $f(a+b)< \min\{f(a),f(b)\}$ for some $a,b\in R$. In this case we have that $f(a+b)< f(a)$ and $f(a+b)< f(b)$. By Property \textbf{(P2)} we have that there exist open sets $U,U',W,W'\in\mathfrak A$ such that $f(a+b)\in U<W\ni f(a)$ and $f(a+b)\in U'<W'\ni f(b)$. Take now
\[
O=\pi_a^{-1}(W)\cap\pi_b^{-1}(W')\cap\pi_{a+b}^{-1}(U\cap U').
\]
Again we have that $f\in O$. If $g\in O$ we have that $g(a+b) <\min\{g(a),g(b)\}$ which implies that $g$ is not a valuation.

Finally, if \textbf{(V3)} does not hold, then $f(1)\neq 0$ or $f(0)\neq \infty$. Assume that $f(1)\neq 0$. Since $\mathfrak A$ is Hausdorff the set $\Gamma_\infty\setminus\{0\}$ is open. Take the set $O=\pi_1^{-1}(\Gamma_\infty\setminus\{0\})$. Then $f\in O$ and $O\cap \mathcal{W}_{\Gamma'}=\emptyset$. The case of $f(0)\neq\infty$ is treated analogously.
\end{proof}

Before proving Proposition \ref{Proptop}, we observe that since $R$ is not an algebraic extension of $\F_p$ for some prime number $p$, for every $\gamma\in\Gamma\setminus\{0\}$, there exists a valuation $\nu:R\lra \Gamma_\infty$ and $a_0\in R$ such that $\nu(a_0)=\gamma$.

\begin{proof}[Proof of Proposition \ref{Proptop}]
Suppose $\mathfrak A$ is not $T_1$ and let us prove that there exists $f\in (\Gamma_\infty )^R$ with the following property: for any open set $O\subseteq \left(\Gamma_\infty\right)^R$ if $f\in O$, then $O\cap \mathcal W_\Gamma\neq\emptyset$.

Since $\mathfrak A$ is not $T_1$ there exist elements $\gamma,\gamma'\in\Gamma_\infty,\gamma\neq\gamma'$ such that for every $U\in\mathfrak A$, if $\gamma'\in U$, then $\gamma\in U$. If $\gamma=0$ , then we just take a valuation $\nu$ and define the function
\begin{displaymath}
f(a)=
\left\{\begin{array}{l}
\nu(a)\ ,\textnormal{ if } a \neq 1\\
\gamma'\ \ \ \   ,\textnormal{ if } a =1.
\end{array}\right.
\end{displaymath}
It is easy to see that $f$ is not a valuation and that every open set that contains $f$ contains $\nu$. The case $\gamma=\infty$ is treated analogously.

If $\gamma\neq 0$ and $\gamma\neq\infty$ we take a valuation $\nu:R\lra\Gamma_\infty$ such that $\nu(a_0)=\gamma$ for some $a_0\in R$. Define now the function
\begin{displaymath}
f(a)=\left\{\begin{array}{l}
\nu(a)\ ,\mbox{ if } a \neq a_0 \\
\gamma'\ \ \ \ ,\mbox{ if } a =a_0.
\end{array}
\right.
\end{displaymath}
Since $a_0^2\neq a_0$ we have that
\[
f(a_0^2)=\nu(a_0^2)=2\nu(a_0)=2\gamma\neq 2\gamma'=2f(a_0).
\]
Hence, $f$ is not a valuation. Take any open set $O$ in the product topology such that $f\in O$. For $a\in R$, if $f(a)\neq \nu(a)$, then $f(a)=\gamma$ and $\nu(a)=\gamma'$. Therefore, $\nu\in O$ which is what we wanted to prove.
\end{proof}

\section{Topologies on $\Gamma_\infty$}

We start this section by presenting some natural topologies on $\Gamma_\infty$. Since $\Gamma$ is a totally ordered set, it carries a topology induced by the order. There are (at least) three natural ways to extend this topology to $\Gamma_\infty$.

\begin{Def}\label{Ordtopo}
For a totally ordered set $X$, the \textbf{order topology} is defined as the topology generated by the sets of the form
\[
\{x\in X\mid x>x_0\}\mbox{ and }\{x\in X\mid x<x_0\}
\]
where $x_0$ runs through $X$. We denote by $X_\infty$ the set $X\cup\{\infty\}$ where $\infty$ is an element not belonging to $X$ and extend the order from $X$ to $X_\infty$ by setting $\infty>x$ for every $x\in X$. In this manner, $X_\infty$ is a totally ordered set and hence we can talk about the order topology on $X_\infty$. A system of open neighbourhoods of $\infty$ in this topology is given by
\[
\{x\in X_\infty\mid x>x_0\},
\]
with $x_0$ running through $X$.
\end{Def}

\begin{Def}
Take a totally ordered set $X$ and an element $y\notin X$. Define the \textbf{circle topology} on $X'=X\cup\{y\}$ as follows: consider the order topology on $X$ and extend it to $X'$ by taking
\begin{equation}\label{opbascirtop}
\{y\}\cup\{x\in X\mid (x<x_0)\vee (x>x_1)\}=\{y\}\cup X\setminus [x_0,x_1],
\end{equation}
as a system of open neighbourhoods of $y$, where $x_0$ and $x_1$ run through $X$.
\end{Def}

\begin{Def}
Take any topological space $(X,\mathfrak A)$. The \textbf{one point compactification of $(X,\mathfrak A)$} is the topological space given by $(X',\mathfrak A')$ where $X'=X\cup\{y\}$ with $y\notin X$ and
\[
\mathfrak A'=\mathfrak A\cup\{\{y\}\cup \left(X\setminus K\right)\mid K\textnormal{ is closed and compact in }(X,\mathfrak A)\}.
\]
\end{Def}

\begin{Def}
We define the topologies $\mathfrak A_1$, $\mathfrak A_2$ and $\mathfrak A_3$ on $\Gamma_\infty$ as the order, the circle and the one point compactification topologies, respectively. Here the topology on $\Gamma$ is the order topology and we set $\infty$ to be the element denoted by $y$ in the previous definitions.
\end{Def}

\begin{Obs}
Since the order topology on $\Gamma$ is Hausdorff, the open sets of $\mathfrak A_3$ which contain $\infty$ are $\{\infty\}\cup \left(X\setminus K\right)$ where $K$ is a compact subset of $\Gamma$ (in a Hausdorff space every compact set is closed).
\end{Obs}

\begin{Lema}
For an ordered abelian group $\Gamma$, the order topology on $\Gamma$ is discrete if and only if $\Gamma$ has a smallest positive element.
\end{Lema}

\begin{proof}
If $\Gamma$ has a smallest positive element $\alpha$, then $\left]0,\alpha\right[=\emptyset$. Then $\left]\alpha,2\alpha\right[=\emptyset$, because if $\alpha_1\in\left]\alpha,2\alpha\right[$ we would have $\alpha_1-\alpha\in\left]0,\alpha\right[$. Therefore, $\{\alpha\}=\left]0,2\alpha\right[$ is open in the order topology. For any element $\beta\in \Gamma$ we have that $\{\beta\}=\left]\beta-\alpha,\beta+\alpha\right[$ is open in the order topology. Indeed, if
\[
\beta'\in\left]\beta-\alpha,\beta+\alpha\right[\textrm{ and }\beta\neq\beta'
\]
we would have that
\[
\beta'-\beta+\alpha\in\left]0,2\alpha\right[\textrm{ and }\beta'-\beta+\alpha\neq \alpha
\]
which is a contradiction. Therefore, the order topology is discrete.

Suppose that the order topology is discrete. This implies that $\{0\}$ is an open set, hence there exist $\alpha>0$ and $\beta<0$ such that
\[
\{\gamma\in\Gamma\mid \beta<\gamma<\alpha\}\subseteq \{0\}.
\]
Therefore, $\alpha$ is the smallest positive element for $\Gamma$.
\end{proof}

\begin{Def}
A totally ordered set $X$ is said to be \index{Complete!totally ordered set}\textbf{complete} if every non-empty set bounded from above admits a supremum.
\end{Def}

\begin{Prop}\label{12}
Take a totally ordered set $X$. We have the following:
\begin{description}
\item[(i)] If the order topology on $X_\infty$ is compact, then $X$ has a smallest element.

\item[(ii)] The order topology on $X_\infty$ is finer than or equal to the circle topology on $X_\infty$. Moreover, these topologies are equal if and only if $X$ has a smallest element.

\item[(iii)] The circle topology on $X_\infty$ is finer than or equal to the one point compactification on $X_\infty$. Moreover, they are equal if and only if $X$ is complete.

\end{description}
\end{Prop}

\begin{proof}
\textbf{(i)} Suppose that $X$ does not have smallest element. Then the family $\{U_x\}_{x\in X}$, with $U_x=\left]x,\infty\right]$, is an open covering of $X_\infty$ in the order topology. Also, for every finite subfamily $\{U_{x_i}\}_{1\leq i\leq n}$ of $\{U_x\}_{x\in X}$, we have that
\[
X_\infty\neq U_{x_0}=\bigcup_{i=1}^nU_{x_i},
\]
where $\displaystyle x_0=\min_{1\leq i\leq n}x_i$. Therefore, the order topology is not compact.

\textbf{(ii)} Since both topologies extend the order topology of $X$, we just have to consider neighbourhoods of $y$. A subbasic open neighbourhood of $\infty$ in the circle topology is of the form $U_1\cup U_2$ where
\[
U_1=\{x\in X_\infty\mid x<x_0\}\mbox{ and }U_2=\{x\in X_\infty\mid x>x_1\},
\]
for some $x_0,x_1\in X$. Since both $U_1$ and $U_2$ are open sets in the order topology, so is $U_1\cup U_2$.

Assume now that $X$ has smallest element $x'$. Then every subbasic open neighbourhood of $\infty$ in the order topology is of the form
\[
\{x\in X_\infty\mid x>x_0\}=\{x\in X_\infty\mid x<x'\}\cup\{x\in X_\infty\mid x>x_0\},
\]
which is open in the circle topology. On the other hand, if $X$ does not have a smallest element, then for $x_0\in X$ the set $U=\{x\in X_\infty\mid x>x_0\}$ is open in the order topology, but not in the circle topology.

\textbf{(iii)} Take any open subset $U$ of $X_\infty$ in the one point compactification. For every $x\in U$ we have to show that there exists an open set $V$ in the circle topology such that $x\in V\subseteq U$. Since the case is trivial for $x\neq \infty$ we suppose that $x=\infty$ and $U=\{\infty\}\cup \Gamma\setminus K$ such that $K$ is compact in $X$. Since $K$ is compact, it must be bounded, i.e., $K\subseteq \left[x_0,x_1\right]$ for some $x_0,x_1\in X$. Therefore, 
\[
\infty\in V:=\{\infty\}\cup X\setminus \left[x_0,x_1\right]\subseteq\{\infty\}\cup X\setminus K=U.
\]

Assume now that $X$ is not complete. Then there exists a non-empty subset $S$ of $X$, bounded from above but without a supremum. Hence, for every $x$ such that $x\geq S$ there exists $x_1\in X$ such that $S\leq x_1<x$. Define the following family of open sets in the circle topology:
\[
\mathcal{F}=\{U_{x_0}^{x_1}\mid x_0\in S\mbox{ and }S\leq x_1\},\mbox{ where }U_{x_0}^{x_1}=\{\infty\}\cup X\setminus [x_0,x_1].
\]
Take an element $x\in X$. If $x\geq S$, then there exists $x_1\in X$ such that $S\leq x_1<x$. Then
$x\in U_{x_0}^{x_1}$ for any $x_0\in S$. If $x\ngeq S$ there exists $x_0\in S$ such that $x< x_0$. Again, $x\in U_{x_0}^{x_1}$ for each $x_1\geq S$. Therefore,
\[
X_\infty=\bigcup_{x_0\in S,x_1\geq S}U_{x_0}^{x_1}=\bigcup\mathcal F.
\]
It is easy to see that for every finite subfamily $\mathcal F'$ of $\mathcal F$ we have that $X_\infty\neq \bigcup\mathcal F'$. Therefore, the circle topology is not compact and hence distinct from the one point compactification.

It remains to show that if $X$ is complete, then the one point compactification and the circle topologies are equal. In view of (\ref{opbascirtop}), it is enough to show that every subset of the form $\left[x_0,x_1\right]$ is compact in $X$ with respect to the order topology. Take any open covering $\{U_i\}_{i\in I}$ of $\left[x_0,x_1\right]$ and consider the set
\[
\mathcal S=\left\{x\in \left[x_0,x_1\right]\mid \exists i_1,\cdots,i_n\in I\textnormal{ such that }\left[x_0,x\right]\subseteq\bigcup_{j=1}^n U_{i_j}\right\}.
\]
This set is non-empty and bounded, so it admits a supremum $x'$. Since $\mathcal S\leq x_1$ we must have $x'\leq x_1$. Suppose towards a contradiction that $x'<x_1$. If $x'$ has an immediate successor $x''$, then we take any $i_{n+1}\in I$ such that $x''\in U_{i_{n+1}}$. Consequently,
\[
\left[x_0,x''\right]=[x_0,x']\cup\{x''\}\subseteq\bigcup_{j=1}^{n+1} U_{i_j}
\]
which implies that $x''\in \mathcal S$. This is a contradiction with $x'=\sup \mathcal S$. If $x'$ does not have an immediate successor we take any $i_{n+1}\in I$ such that $x'\in U_{i_{n+1}}$. Since $U_{i_{n+1}}$ is open in the order topology there exists $x''>x'$ such that
\[
\left[x_0,x''\right]\subseteq \bigcup_{j=1}^{n+1} U_{i_j},
\]
which gives the desired contradiction.
\end{proof}

Using the proposition above and the well-known fact that every complete ordered abelian group is isomorphic to the real numbers we obtain the following:

\begin{Cor}
We have that $\mathfrak A_3\subseteq\mathfrak A_2\subsetneq\mathfrak A_1$. Moreover, $\mathfrak A_3=\mathfrak A_2$ if and only if $\Gamma$ is isomorphic to the real numbers.
\end{Cor}

\begin{Lema}\label{11}
We have the following:
\begin{description}
\item[(a)] Properties \mbox{\rm\textbf{(P1)}} and \mbox{\rm\textbf{(P2)}} hold for $\mathfrak A_1$;
\item[(b)] Properties \mbox{\rm\textbf{(P1)}} and \mbox{\rm\textbf{(P2)}} are satisfied neither by $\mathfrak A_2$ nor by $\mathfrak A_3$.
\end{description}
\end{Lema}

\begin{proof}
\textbf{(a)} Take $\gamma,\gamma'\in \Gamma$ such that $\gamma<\gamma'$. If there is an element $\alpha\in\left]\gamma,\gamma'\right[$ we take
\[
U=\left]-\infty,\alpha\right[\textrm{ and }U'=\left]\alpha,\infty\right[.
\]
If $\left]\gamma,\gamma'\right[=\emptyset$ we take
\[
U=\left]-\infty,\gamma'\right[\textrm{ and }U'=\left]\gamma,\infty\right[.
\]
In each case we have that $\gamma\in U<U'\ni\gamma'$. Therefore, \textbf{(P2)} holds for $\mathfrak A_1$.

In order to show that \textbf{(P1)} holds we must show that for any $\gamma,\gamma'\in\Gamma_\infty$ and $U\in\mathfrak A_1$, if $\gamma+\gamma'\in U$, then there exist $V,V'\in\mathfrak A_1$ with $\gamma\in V$ and $\gamma'\in V'$ such that $V+V'\subseteq U$.

First, consider the case where $\gamma\neq\infty\neq\gamma'$. If the order topology is discrete we just take $V=\{\gamma\}$ and $V'=\{\gamma'\}$. In the other case, take $\alpha,\beta\in\Gamma$ with $\alpha,\beta>0$ such that
\[
\gamma+\gamma'\in \left]\gamma+\gamma'-\alpha,\gamma+\gamma'+\beta\right[\subseteq U.
\]
There exist $\alpha_1,\alpha_2,\beta_1,\beta_2\in\Gamma$ such that
\[
\alpha_1,\alpha_2,\beta_1,\beta_2>0\textrm{ and }\alpha_1+\alpha_2=\alpha\textrm{ and }\beta_1+\beta_2=\beta.
\]
Consider now the open sets
\[
V=\left]\gamma-\alpha_1,\gamma+\beta_1\right[\textrm{ and }V'=\left]\gamma'-\alpha_2,\gamma'+\beta_2\right[.
\]
Then,
\[
V+V'\subseteq \left]\gamma+\gamma'-\alpha_1-\alpha_2,\gamma+\gamma'+\beta_1+\beta_2\right[\subseteq U.
\]

It remains to prove that given any open neighbourhood $U$ of $\infty$ and any $\gamma\in\Gamma_\infty$ there exist $V,V'\in\mathfrak A_1$ with $\infty\in V$ and $\gamma\in V'$ such that $V+V'\in U$. Since $U$ is a neighbourhood of $\infty$ there exists $\alpha\in \Gamma$ such that $\{\alpha'\in\Gamma_\infty\mid \alpha'>\alpha\}\subseteq U$. If $\gamma=\infty$ we just take
\[
V=\{\alpha'\in\Gamma_\infty\mid\alpha'>\alpha\} \textnormal{ and }V'=\{\alpha'\in\Gamma_\infty\mid\alpha'>0\}
\]
and if $\gamma\neq\infty$ we just take any $\beta>0$ and define $V=\{\alpha'\in\Gamma_\infty\mid\alpha'>\alpha-\gamma+\beta\}$ and $V'=\{\alpha'\in\Gamma_\infty\mid\alpha'>\gamma-\beta\}$. In any case we have that $V+V'\subseteq U$.

\textbf{(b)} To prove that \textbf{(P1)} does not hold for $\mathfrak A_2$ and $\mathfrak A_3$ we just have to observe that in each case the set
\[
U=\{\gamma\in\Gamma_\infty\mid\gamma\neq 0\}
\]
is an open neighbourhood of $\infty$. On the other hand, if $V,V'$ are open neighbourhoods of $\infty$, then there exists $\gamma\in\Gamma$ such that $\gamma\in V$ and $-\gamma\in V'$. Therefore, $V+V'\not\subseteq U$.
 
Take an element $\gamma\in \Gamma$ (hence $\gamma<\infty$ in $\Gamma_\infty$) and an open neighbourhood $U$ of $\infty$ in either $\mathfrak A_2$ or $\mathfrak A_3$. From the definition of the topologies $\mathfrak A_2$ and $\mathfrak A_3$ we have that there exists an element $\gamma'\in U$ such that $\gamma'<\gamma$. Therefore, Property \textbf{(P2)} cannot hold for $\mathfrak A_2$ or $\mathfrak A_3$.
\end{proof}

As a consequence of Theorem \ref{topvaltheo} and Lemma \ref{11} we obtain:
\begin{Cor}\label{Cor_4.3}
The set $\mathcal{W}_\Gamma$ is closed in $\left(\Gamma_\infty\right)^{R}$ if we take the order topology $\mathfrak A_1$ on $\Gamma_\infty$.
\end{Cor}

If \textbf{(P2)} holds, then $\mathfrak A$ is finer than the order topology. Indeed, take any $\gamma\in\Gamma_\infty$. For every $\gamma'<\gamma$ there exists an open set $U_{\gamma'}\in \mathfrak A$ such that $U_{\gamma'}<\gamma$. Therefore,
\[
\{\gamma'\in\Gamma_\infty\mid \gamma'<\gamma\}=\displaystyle\bigcup_{\gamma'<\gamma} U_{\gamma'}\in \mathfrak A.
\]
The case $\gamma'>\gamma$ is treated analogously. This proves that every subbasic open set of the order topology belongs to $\mathfrak A$, i.e., this topology is finer than the order topology.

We present now an example that shows that the property of being $T_1$ does not characterize the order topology among its coarser topologies.
\begin{Exa}
Take any element $\gamma\in\Gamma$. Define the topology $\mathfrak A\subsetneq\mathfrak A_1$ on $\Gamma_\infty$ by
\[
\mathfrak A=\{U\in\mathfrak A_1\mid (\gamma\notin U)\vee (\exists\gamma_1,\gamma_2\in\Gamma\textrm{ with }\left]-\infty,\gamma_1\right[\cup\left]\gamma_2,\infty\right[\subseteq U)\}.
\]
It is easy to check that this topology is $T_1$ (it can be even proved that $\mathfrak A$ is normal).
\end{Exa}

For some applications, for instance in algebraic geometry, the interesting valuations are those which are centered. This implies that such valuations take values in $\Gamma':=\left(\Gamma^{\geq 0}\right)_\infty$. We define the topologies $\mathfrak A'_1$, $\mathfrak A'_2$ and $\mathfrak A'_3$ on $\Gamma'$ analogously to the topologies $\mathfrak A_1$, $\mathfrak A_2$ and $\mathfrak A_3$ on $\Gamma_\infty$.

As a consequence of Lemma \ref{12} we obtain the following:

\begin{Prop}\label{proptopgamma}
The topologies $\mathfrak A'_1$ and $\mathfrak A'_2$ of $\Gamma'$ are equal. Also, $\mathfrak A'_3\subseteq \mathfrak A'_2$, and $\mathfrak A'_3=\mathfrak A'_2$ if and only if $\Gamma\simeq\R$.
\end{Prop}

\begin{Cor}\label{14}
The set of all centered valuations of a local ring $R$ taking values in $\R_\infty$ is compact with respect to the product topology of $\left[0,\infty\right]^R$. (Here, $[0,\infty]=\{x\in \R_\infty\mid 0\leq x\leq\infty\}$ is endowed with the order topology).
\end{Cor}
\begin{proof}
Theorem \ref{topvaltheo} gives us that the set of valuations on $R$ with values in the monoid $\R_\infty'$ is closed in $\left(\R_\infty'\right)^R=\left[0,\infty\right]^R$ with respect to the product topology. Also, Lemma \ref{12} guarantees that $\left[0,\infty\right]$ is compact, hence $\left[0,\infty\right]^R$ is compact. Since $\mathcal{W}_\R$ is a closed subset of a compact space it is compact.
\end{proof}

We do not know whether Corollary \ref{14} would remain true if we take $\Gamma=\R^n$ where $n>1$, i.e., we do not have an answer for the following problem:

\begin{Pro}
Take a topology $\mathfrak A$ on $\Gamma_\infty$ where $\Gamma=\R^n$. Is the corresponding topology on the set of all non-negative valuations taking values in $\left(\R^n\right)_\infty$ compact?
\end{Pro}

Our criteria developed so far cannot fully answer this question. Since Properties \textbf{(P1)} and \textbf{(P2)} hold for $\mathfrak A_1$ they also hold for $\mathfrak A'_2=\mathfrak A'_1$. Hence, $\mathcal W_{\Gamma'}$ is closed in $\left(\Gamma'\right)^R$ for the topology $\mathfrak A'_2=\mathfrak A'_1$ on $\Gamma'=\left((\R^n)^{\geq 0}\right)_\infty$. However, since $\mathfrak A'_2=\mathfrak A'_1$ is not compact (for $n>1$) we cannot conclude whether $\mathcal W_{\Gamma'}$ is compact or not. 

On the other hand, if we consider the compact topology $\mathfrak A'_3$ of $\Gamma'$, then Properties \textbf{(P1)} and \textbf{(P2)} do not hold for $\mathfrak A'_3$ and again we cannot conclude whether $\mathcal W_{\Gamma'}$ is compact or not. 

This problem, and its partial answers, will be discussed in a forthcoming paper.

\section{Appendix: Known topologies on spaces of valuations}

We start this appendix by describing an approach used by by Favre and Jonsson in \cite{Fav_1}, following Berkovich's ideas (see \cite{Ber}), to define topologies on sets of valuations. Take a noetherian local ring $R$ with maximal ideal $\MI$ and an ordered abelian group $\Gamma$.
\begin{Def}\label{Norval}
For each positive element $\gamma\in\Gamma$ we say that a centered valuation $\nu:R\lra\Gamma_\infty$ is \index{Valuation!normalized}\textbf{normalized by $\gamma$} if $\nu(\MI)=\gamma$.
\end{Def}

\begin{Def}\label{topvalcentnorm}
Consider the subset $\mathcal W_\MI$ of $\mathcal{W}_\R$ consisting of all centered valuations normalized by $1$. We define the weak topology on $\mathcal W_\MI$ to be the topology having as a subbasis the sets of the form
\[
\{\nu\in \mathcal W_\MI\mid \nu(a)>\alpha\}\textnormal{ and }\{\nu\in 
\mathcal W_\MI\mid \nu(a)<\alpha\}
\]
where $\alpha$ runs through $\R_\infty$ and $a$ runs through $R$.
\end{Def}

Observe that this is the restriction to $\mathcal W_\MI$ of the topology appearing in Corollary \ref{14}. It is easy to prove that $\mathcal W_\MI$ is closed in $[0,\infty]^R$ and hence the topology defined above is compact.

\begin{Obs}
If $R$ is a two-dimensional regular local ring, then $\mathcal W_\MI$ admits a tree structure, which will be called the valuative tree of $R$.
\end{Obs}

\begin{Obs}
The sets of the form
\[
\{r\in\R_\infty\mid r>\alpha\}\mbox{ and }\{r\in\R_\infty\mid r<\alpha\}\mbox{ with $\alpha$ running through }\R_\infty
\]
form a subbasis of open sets for the order topology on $\R_\infty$. The product topology on $\left(\R_\infty\right)^{R}$ is the weak topology associated to the projections into $\R_\infty$. Hence, the topology defined in \ref{topvalcentnorm} is the topology on $\mathcal W_\MI\subseteq \left(\R_\infty\right)^{R}$ induced by the product topology on $\left(\R_\infty\right)^{R}$.
\end{Obs}

The tree structure in $\mathcal W_\MI$ (for a two dimensional regular local ring) endows some interesting topologies in $\mathcal W_\MI$. For a more detailed discussion on that topic, we suggest \cite{Fav_1} and \cite{Nov1}.

So far we have been constructing topologies on sets of valuation. From now on, we will identify equivalent valuations. We fix an integral domain $R$ and denote by $\mathcal W$ and $\mathcal V$ the sets of valuations and Krull valuations, respectively, on $R$.

This means that the topologies will now be defined in the set of classes of valuations. We observe that any topology on the set of valuations induces canonically a topology on the classes of valuations (namely, the quotient topology).

\begin{Def}\label{Zar_top}
The \index{Topology!Zariski}\textbf{Zariski topology} on $\mathcal{V}$ is the topology generated by the sets of the form
\[
\{\nu\in \mathcal{V}\mid \nu(a)\geq 0\}
\]
where $a$ runs through $F=\QF(R)$. 
\end{Def}

The coarsest Hausdorff topology on $\mathcal{V}$ which is finer than the Zariski topology is the patch (also called constructive) topology:

\begin{Def}
The \index{Topology!patch} \textbf{patch topology} on $\mathcal{V}$ is defined to be the topology generated by the sets of the form
\[
\{\nu\in \mathcal{V}\mid \nu(a)\geq 0\}\mbox{ and }\{\nu\in \mathcal{V}\mid \nu(b)>0\}
\]
where $a$ and $b$ run through $F$.
\end{Def}

\begin{Def}
Take a set $X$ and a family $\mathcal F=\left\{(X_i,\mathfrak{A}_i,\Phi_i)\mid i\in I\right\}$ where for every $i\in I$, $(X_i,\mathfrak{A}_i)$ is a topological space and $\Phi_i:X\lra X_i$ is a function. We define the \index{Topology!weak}\textbf{weak topology on $X$ associated to $\mathcal F$} to be the coarsest topology which makes all the $\Phi_i$ continuous. It is equivalent to say that this topology is the topology having as a subbasis all sets of the form $\Phi_i^{-1}(U_i)$ with $U_i\in \mathfrak A_i$ and $i\in I$.

\end{Def}

For every $a\in F=\QF(R)$ let
\[
\pi_a:\mathcal{V}\lra \{0,-,+\}
\]
be the function given by
\[
\pi_a(\nu)=\begin{cases}
0&\textnormal { if }\nu(a)=0,\\ -& \textnormal{ if }\nu(a)<0,\\ + & \textnormal{ if } \nu(a)>0.
\end{cases}
\]
Endow $X:=\{0,-,+\}$ with the topologies
\[
\mathfrak A_1:=\{\emptyset,\{0,+\}, X\}\textnormal{ and }\mathfrak A_2:=\{\emptyset,\{+\},\{0,+\}, X\}.
\]
Then the Zariski topology is the weak topology on $\mathcal V$ induced by $\left\{(X,\mathfrak A_1, \pi_a)\mid a\in F\right\}$ and the patch topology is the weak topology on $\mathcal V$ induced by $\left\{(X,\mathfrak A_2, \pi_a)\mid a\in F\right\}$.

Since $X$ is finite, both topologies $\mathfrak A_1$ and $\mathfrak A_2$ are compact. Hence, in order to obtain the compactness of the Zariski and patch topologies, it is enough to show that the set of equivalence classes is closed in $X^F$. This is proved for instance in \cite{Zar}.

The way that the Zariski topology is constructed on $\mathcal V$ (Definition \ref{Zar_top}) cannot be applied to the set $\mathcal W$ of all equivalence classes of valuations. That is because $R$ does not need to be a domain. Even if $R$ is a domain we cannot guarantee that every valuation on $R$ can be extended to $F=\QF(R)$. To overcome this problem, Huber and Knebusch introduced the valuation spectrum topology (see \cite{Kne}).
\begin{Def}\label{spectopspval}
We define the \index{Topology!valuation spectrum}\textbf{valuation spectrum topology} on $\mathcal{W}$ as the topology having as a subbasis the sets
\[
\{\nu\in\mathcal{W}\mid \nu(a)\geq\nu(b)\neq\infty\}
\]
where $a$ and $b$ run through $R$.
\end{Def}

One can see that the restriction of the valuation spectrum topology from $\mathcal W$ to $\mathcal V$ is the Zariski topology.

\end{document}